\documentclass[10pt]{article}
\usepackage{amsfonts}
\usepackage{epsfig,epic}

\usepackage{color}
\usepackage{ulem}

\usepackage{amsmath}
\usepackage{amsfonts}
\usepackage{amssymb,amsthm}

\newcommand{\calZ}{{\mathcal Z}}

\newcommand{\calH}{{\mathcal H}}

\newcommand{\N}{{\mathbb N}}

\newcommand{\Z}{{\mathbb Z}}
\newcommand{\Q}{{\mathbb Q}}
\newcommand{\R}{{\mathbb R}}
\newcommand{\C}{{\mathbb C}}

\newcommand{\Frac}[2]{\displaystyle \frac{#1}{#2}}
\newcommand{\Sum}[2]{\displaystyle{\sum_{#1}^{#2}}}
\newcommand{\Prod}[2]{\displaystyle{\prod_{#1}^{#2}}}
\newcommand{\Lim}[1]{\displaystyle{\lim_{#1}\ }}

\def\pol#1{\langle #1 \rangle}
\def\Lyn{{\mathcal Lyn}}

\def\CX{\C \langle X \rangle}
\def\CY{\C \langle Y \rangle}
\def\abs#1{|#1|}

\def\AXX{\serie{A}{X}}
 \def\shuffle{\mathop{_{^{\sqcup\!\sqcup}}}} 

\gdef\stuffle{\;%
  \setlength{\unitlength}{0.0125cm}%
  \begin{picture}(20,10)(220,580) 
  \thinlines 
  \put(220,592){\line( 0,-1){ 10}} 
  \put(220,582){\line( 1, 0){ 20}} 
  \put(240,582){\line( 0, 1){ 10}} 
  \put(230,592){\line( 0,-1){ 10}} 
  \put(225,587){\line( 1, 0){ 10}} 
  \end{picture}\; 
}

\newtheorem{proposition}{Proposition}
\newtheorem{theorem}{Theorem}
\newtheorem{lemma}{Lemma}

\newtheorem{remark}{Remark}

\newcommand{\Li}{\operatorname{Li}}

\def\L{\mathrm{L}}
\def\H{\mathrm{H}}

\def\deg{\mathrm{deg}}

\newcommand{\poly}[2]{#1 \langle #2 \rangle}

\def\QX{\poly{\Q}{X}}

\def\CX{\poly{\C}{X}}

\def\QY{\poly{\Q}{Y}}

\def\CY{\poly{\C}{Y}}

\newcommand{\serie}[2]{#1 \langle \! \langle #2 \rangle \! \rangle}
\def\AXX{\serie{A}{X}}
\def\AYY{\serie{A}{Y}}

\def\CXX{\serie{\C}{X}}

\def\pol#1{\langle #1 \rangle}

\def\AX{A \langle X \rangle}
\def\AY{A \langle Y \rangle}

\def\Lie{{\cal L}ie}

\def\LAX{\Lie_{A} \langle X \rangle}

\def\LAZ{\Lie_{A} \langle Z \rangle}

\def\LAZZ{\Lie_{A} \langle\!\langle Z \rangle \!\rangle}

\def\CX{\C \langle X \rangle}
\def\CXX{\serie{\C}{X}}

\def\Lim{\displaystyle\lim}
\def\Sum{\displaystyle\sum}
\def\Prod{\displaystyle\prod}
\def\Frac{\displaystyle\frac}
\def\path{\rightsquigarrow}

\def\bv{\mid}

\def\abs#1{\bv\!#1\!\bv}

\gdef\minishuffle{{\scriptstyle \shuffle}}  
\gdef\ministuffle{{\scriptstyle \stuffle}}

\def\scal#1#2{\langle #1\bv#2 \rangle}

\def\deg{\mathop\mathrm{deg}\nolimits}

\def\binom#1#2{{#1\choose#2}}

\begingroup
\def\2#1{\ifnum#1<10 0\fi\the#1}
\xdef\isodayandtime{
{\2\day-\2\month-\the\year\space\2{\count0}:%
\2{\count2}}}
\endgroup

\def\d{\mathbf{d}}

\newcounter{per1}
\begin{document}

\begin{center}
{\Large On Drinfel'd associators}

\medskip
{\large G. H. E. Duchamp$^{\sharp}$, V. Hoang Ngoc Minh$^{\lozenge}$, K. A. Penson$^{\flat}$}

\medskip

$^{\sharp}$Universit\'e Paris XIII, 99 Jean-Baptiste Cl\'ement, 93430 Villetaneuse, France.\\
$^{\lozenge}$Universit\'e Lille II, 1, Place D\'eliot, 59024 Lille, France.\\
$^{\flat}$Universit\'e Paris VI, 75252 Paris Cedex 05, France.\\
\end{center}

\section{Knizhnik-Zamolodchikov differential equations and coefficients of Drinfel'd associators}

In 1986 \cite{drinfeld}, in order to study the linear representations of the braid group $B_n$
coming from the monodromy of the Knizhnik-Zamolodchikov differential equations,
Drinfel'd introduced a class of formal power series $\Phi$
on noncommutative variables over the finite alphabet $X=\{x_0,x_1\}$.
Such a power series $\Phi$ is called an {\it associator}.
For $n=3$, it leads to the following fuchsian noncommutative
differential equation with three regular singularities in $\{0,1,+\infty\}$~:
\begin{eqnarray*}
(DE)&&dG(z)=\biggl(x_0\Frac{dz}z+x_1\Frac{dz}{1-z}\biggr)G(z).
\end{eqnarray*}

Solutions of $(DE)$ are power series, with coefficients which are mono-valued
functions on the simply connex domain $\Omega=\C-(]-\infty,0]\cup[1,+\infty[)$
and can be seen as multi-valued over\footnote{In fact, we have mappings
from the connected covering $\widetilde{\C-\{0,1\}}$.} ${\C-\{0,1\}}$,
on noncommutative variables $x_0$ and $x_1$.
Drinfel'd proved that $(DE)$ admits two particular mono-valued solutions on $\Omega$,
$G_0(z)\;{}_{\widetilde{z\path0}}\;\exp[x_0\log(z)]$ and
$G_1(z)\;{}_{\widetilde{z\path1}}\;\exp[-x_1\log(1-z)]$ \cite{drinfeld1,drinfeld2}.
and the existence of an associator $\Phi_{KZ}\in\serie{\R}{X}$ such that
$G_0=G_1\Phi_{KZ}$ \cite{drinfeld1,drinfeld2}.
After that, via representations of the chord diagram algebras,
L\^e and Murakami \cite{lemurakami} expressed the coefficients of $\Phi_{KZ}$
as {\it linear} combinations of special values of several complex variables
{\it zeta} functions, $\{\zeta_r\}_{r\in\N_+}$,
\begin{eqnarray}
\zeta_r:\calH_r\rightarrow\R,&&
(s_1,\ldots,s_r)\mapsto\sum_{n_1>\ldots>n_k>0}\frac1{n_1^{s_1}\ldots n_k^{s_r}},
\end{eqnarray}
where $\calH_r=\{(s_1,\ldots,s_r)\in\C^r\vert\forall m=1,..,r,\sum_{i=1}^m\Re(s_i)>m\}$.
For $(s_1,\ldots,s_r)\in\calH_r$, one has two ways of thinking $\zeta_r(s_1,\ldots,s_r)$
as limits, fulfilling identities \cite{FPSAC97,SLC44,Bui}.
Firstly, they are limits of {\it polylogarithms} and secondly, as truncated sums,
they are limits of {\it harmonic sums}, for $z\in\C,\abs{z}<1,N\in\N_{+}$~:
\begin{eqnarray}
\Li_{s_1,\ldots,s_k}(z)=\Sum_{n_1>\ldots>n_k>0}\frac{z^{n_1}}{n_1^{s_1}\ldots n_k^{s_k}},
\H_{s_1,\ldots,s_k}(N)=\Sum_{n_1>\ldots>n_k>0}^N\frac1{n_1^{s_1}\ldots n_k^{s_k}}.
\end{eqnarray}
More precisely, if $(s_1,\ldots,s_r)\in\calH_r$ then, after a theorem by Abel, one has
\begin{eqnarray}\label{zetavalues}
\Lim_{z\rightarrow1}\Li_{s_1,\ldots,s_k}(z)=\Lim_{n\rightarrow\infty}\H_{s_1,\ldots,s_k}(n)=:\zeta_r(s_1,\ldots,s_k)
\end{eqnarray}
else it does not hold, for $(s_1,\ldots,s_r)\notin\calH_r$, while $\Li_{s_1,\ldots,s_k}$ is well defined
over $\{z\in\C,\abs{z}<1\}$ and so is $\H_{s_1,\ldots,s_k}$, as Taylor coefficients of the following function
\begin{eqnarray}\label{gs}
(1-z)^{-1}{\Li_{s_1,\ldots,s_k}(z)}=\sum_{n\ge1}\H_{s_1,\ldots,s_k}(n)z^n,&\mbox{for}&z\in\C,\abs{z}<1.
\end{eqnarray}

Note also that, for $r=1$, $\zeta_1$ is nothing else
but the famous Riemann zeta function and, for $r=0$,
it is convenient to set $\zeta_0$ to the constant function
$1_{\R}$. In all the sequel, for simplification, we will
adopt the notation $\zeta$ for $\zeta_r,r\in\N$.

In this work, we will describe the regularized solutions of $(DE)$.

For that, we are considering the alphabets $X=\{x_0,x_1\}$ and $Y_0=\{y_s\}_{s\ge0}$
equipped of the total ordering $x_0<x_1$ and $y_0>y_1>y_2>\ldots$, respectively.
Let $Y=Y_0-\{y_0\}$.
The free monoid generated by $X$ (resp. $Y,Y_0$) is denoted by $X^*$
(resp. $Y^*,Y_0^*$) and admits $1_{X^*}$ (resp. $1_{Y^*},1_{Y_0^*}$) as unit.

The sets of, respectively, polynomials and formal power series,
with coefficients in a commutative $\Q$-algebra $A$,
over $X^*$ (resp. $Y^*,Y_0^*$) are denoted by $\AX$
(resp. $\AY,A\pol{Y_0}$) and $\AXX$ (resp. $\AYY,\serie{A}{Y_0}$).
The sets of polynomials are the $A$-modules and endowed with
the associative concatenation, the associative commutative shuffle
(resp. quasi-shuffle) product, over $\AX$ (resp. $\AY,A\pol{Y_0}$).
Their associated coproducts are denoted, respectively, $\Delta_{\shuffle}$
and $\Delta_{\stuffle}$. The algebras $(\AX,\shuffle,1_{X^*})$ and $(\AY,\stuffle,1_{Y^*})$
admit the sets of Lyndon words denoted, respectively, by $\Lyn X$ and $\Lyn Y$,
as transcendence bases \cite{reutenauer} (resp. \cite{acta,VJM}).

For $Z=X$ or $Y$, denoting $\LAZ$ and $\LAZZ$ the sets of, respectively,
Lie polynomials and Lie series, the enveloping algebra $\mathcal{U}(\LAZ)$
is isomorphic to the Hopf algebra $(\AX,.,1_{Z^*},\Delta_{\shuffle},{\tt e})$.
We get also $\calH_{\stuffle}:=(\AY,.,1_{Y^*},\Delta_{\stuffle},
{\tt e})\cong\mathcal{U}(\mathrm{Prim}(\calH_{\ministuffle}))$, where \cite{acta,VJM}
\begin{eqnarray}
\mathrm{Prim}(\calH_{\stuffle})&=&\mathrm{span}_{A}\{\pi_1(w)\vert{w\in Y^*}\},\\
\pi_1(w)&=&\sum_{k=1}^{(w)}\frac{(-1)^{k-1}}k
\sum_{u_1,\ldots,u_k\in Y^+}\scal{w}{u_1\stuffle\ldots\stuffle u_k}u_1\ldots u_k.\label{pi1}
\end{eqnarray}

\section{Indexing polylogarithms and harmonic sums by words and their generating series}\label{polylogarithms}

For any $r\in\N$, since any combinatorial composition $(s_1,\ldots,s_r)\in\N_+^r$ can be associated with words
$x_0^{s_1-1}x_1\ldots x_0^{s_r-1}x_1\in X^*x_1$ and $y_{s_1}\ldots y_{s_r}\in Y^*$.
Similarly, any multi-indice\footnote{The weight of $(s_1,\ldots,s_r)\in\N_+^r$ (resp. $\N^r$) is defined
as the integer $s_1+\ldots+s_r$ which corresponds to the weight, denoted $(w)$, of its associated word $w\in Y^*$
(resp. $Y_0^*$) and corresponds also to the length, denoted by $|u|$, of its associated word $u\in X^*$.}
$(s_1,\ldots,s_r)\in\N^r$ can be associated with words $y_{s_1}\ldots y_{s_r}\in Y_0^*$.
Then let $\Li_{x_0^r}(z):={(\log(z))^r}/{r!}$, and
$\Li_{s_1,\ldots,s_k}$ and $\H_{s_1,\ldots,s_k}$ be indexed by words \cite{FPSAC97}~:
$\Li_{x_0^{s_1-1}x_1\ldots x_0^{s_r-1}x_1}:=\Li_{s_1,\ldots,s_r}$
and $\H_{y_{s_1}\ldots y_{s_r}}:=\H_{s_1,\ldots,s_r}$.
Similarly, $\Li_{-s_1,\ldots,-s_k}$ and $\H_{-s_1,\ldots,-s_k}$ be indexed by words\footnote{Note that, all these
$\{\Li^-_w\}_{w\in Y_0^*}$ and $\{\H^-_w\}_{w\in Y_0^*}$ are divergent at their singularities.} \cite{Ngo,Ngo2}~:
$\Li^-_{y_{s_1}\ldots y_{s_r}}:=\Li_{-s_1,\ldots,-s_r}$
and $\H^-_{y_{s_1}\ldots y_{s_r}}:=\H_{-s_1,\ldots,-s_r}$.
In particular, $\H^-_{y_0^r}(n):=\binom{n}{r}=(n)_r/r!$ and $\Li^-_{y_0^r}(z):=({z}/(1-z))^r$.
There exists a law of algebra, denoted by $\top$, in $\serie{\Q}{Y_0}$,
such that he following morphisms of algebras are {\em surjective} \cite{Ngo}
\begin{eqnarray}
\H^-_{\bullet}:(\Q\pol{Y_0}, \stuffle,1_{Y_0^*})\longrightarrow(\Q\{\H^-_w\}_{w\in Y_0^*},\times,1),&&w\longmapsto\H^-_w,\\
\Li^-_{\bullet}:(\Q\pol{Y_0},\top,1_{Y_0^*})\longrightarrow(\Q\{\Li^-_w\}_{w\in Y_0^*},\times,1),&&w\longmapsto\Li^-_w,
\end{eqnarray}
and $\ker\H^-_{\bullet}=\ker\Li^-_{\bullet}=\Q\langle\{w-w\top 1_{Y_0^*}|w\in Y_0^*\}\rangle$ \cite{Ngo}.
Moreover, the families $\{\H^-_{y_k}\}_{k\ge0}$ and $\{\Li^-_{y_k}\}_{k\ge0}$ are $\Q$-linearly independent.

On the other hand, the following morphisms of algebras are {\em injective} 
\begin{eqnarray}
\H_{\bullet}:(\QY,\stuffle,1_{Y^*})\longrightarrow(\Q\{\H_w\}_{w\in Y^*},\times,1),&&w\longmapsto\H_w,\label{H}\\
\Li_{\bullet}:(\QX,\shuffle,1_{X^*})\longrightarrow(\Q\{\Li_w\}_{w\in X^*},\times,1),&&w\longmapsto\Li_w\label{Li}
\end{eqnarray}
Moreover, the families $\{\H_w\}_{w\in Y^*}$ and $\{\Li_w\}_{w\in X^*}$ are $\Q$-linearly independent
and the families $\{\H_l\}_{l\in\Lyn Y}$ and $\{\Li_l\}_{l\in\Lyn X}$are $\Q$-algebraically independent.
But at singularities of $\{\Li_w\}_{w\in X^*},\{\H_w\}_{w\in Y^*}$, the following convergent values
\begin{eqnarray}\label{cvg}
\forall u\in Y^*-y_1Y^*,\zeta(u):=\H_u(+\infty)
&\mbox{and}&\forall v\in x_0X^*x_1,\zeta(v):=\Li_v(1)
\end{eqnarray}
are no longer linearly independent and the values
$\{\H_l(+\infty)\}_{l\in\Lyn Y-\{y_1\}}$ (resp. $\{\Li_l(1)\}_{l\in\Lyn X-X}$)
are no longer algebraically independent \cite{FPSAC97,zagier}.

The graphs of the isomorphisms of algebras, $\Li_{\bullet}$ and $\H_{\bullet}$,
as generating series, read then \cite{JSC,FPSAC97}
\begin{eqnarray}
\L:=\sum_{w\in X^*}\Li_ww=\prod_{l\in\Lyn X}^{\searrow}e^{\Li_{S_l}P_l},&&
\H:=\Sum_{w\in Y^*}\H_ww=\prod_{l\in\Lyn Y}^{\searrow}e^{\H_{\Sigma_l}\Pi_l},
\end{eqnarray}
where the PBW basis $\{P_w\}_{w\in X^*}$ (resp. $\{\Pi_w\}_{w\in Y^*}$)
is expanded over the basis of $\LAX$ (resp. $\mathcal{U}(\mathrm{Prim}(\calH_{\stuffle}))$,
$\{P_l\}_{l\in \Lyn X}$ (resp. $\{\Pi_l\}_{l\in\Lyn Y}$), and
$\{S_w\}_{w\in X^*}$ (resp. $\{\Sigma_w\}_{w\in Y^*}$) is the basis of
$(\QY,\shuffle,1_{X^*})$ (resp. $(\QY,\stuffle,1_{Y^*})$) containing
the transcendence basis $\{S_l\}_{l\in \Lyn X}$ (resp. $\{\Sigma_l\}_{l\in\Lyn Y}$).

By termwise differentiation, $\L$ satisfies the noncommutative
differential equation $(DE)$ with the boundary condition
$\L(z){}_{\widetilde{z\to0^+}}e^{x_0\log(z)}$.
It is immediate that the power series $\H$ and $\L$ are group-like,
for $\Delta_{\stuffle}$ and $\Delta_{\shuffle}$, respectively. Hence,
the following noncommutative generating series are well defined and are group-like,
for $\Delta_{\stuffle}$ and $\Delta_{\shuffle}$, respectively \cite{FPSAC97,acta,VJM}~:
\begin{eqnarray}
Z_{\stuffle}:=\prod_{l\in\Lyn Y-\{y_1\}}^{\searrow}e^{\H_{\Sigma_l}(+\infty)\Pi_l}
&\mbox{and}&Z_{\shuffle}:=\prod_{l\in\Lyn X-X}^{\searrow}e^{\Li_{S_l}(1)P_l}.
\end{eqnarray}

Definitions \eqref{zetavalues} and \eqref{cvg} lead then to the following {\em surjective} poly-morphism
\newcommand{\longtwoheadrightarrow}{\ensuremath{\joinrel\relbar\joinrel\twoheadrightarrow}}
\begin{eqnarray}\label{zeta}
\zeta:{\displaystyle(\Q1_{X^*}\oplus x_0\QX x_1,\shuffle,1_{X^*})\atop
\displaystyle
(\Q1_{Y^*}\oplus(Y-\{y_1\})\QY,\stuffle,1_{Y^*})}&\longtwoheadrightarrow&(\calZ,\times,1),\\
{\displaystyle x_0x_1^{r_1-1}\ldots x_0x_1^{r_k-1}\atop
\displaystyle y_{s_1}\ldots y_{s_k}}&\longmapsto&\Sum_{n_1>\ldots>n_k>0}{n_1^{-s_1}\ldots n_k^{-s_k}},
\end{eqnarray}
where $\calZ$ is the $\Q$-algebra generated by $\{\zeta(l)\}_{l\in\Lyn X-X}$
(resp. $\{\zeta(S_l)\}_{l\in\Lyn X-X}$),
or equivalently, generated by $\{\zeta(l)\}_{l\in\Lyn Y-\{y_1\}}$
(resp. $\{\zeta(\Sigma_l)\}_{l\in\Lyn Y-\{y_1\}}$).

Now, let $t_i\in\C,\abs{t_i}<1,i\in\N$. For $z\in\C,\abs{z}<1$, we have \cite{IMACS}
\begin{eqnarray}\label{simple}
\Sum_{n\ge0}\Li_{x_0^n}(z)\;t_0^n=z^{t_0}&\mbox{and}&\sum_{n\ge0}\Li_{x_1^n}(z)\;t_1^n=(1-z)^{-t_1}.
\end{eqnarray}
These suggest to extend the morphism $\Li_{\bullet}$ over
$(\mathrm{Dom}(\Li_{\bullet}),\shuffle,1_{X^*})$,
via {\it Lazard's elimination}, as follows (subjected to be convergent)
\begin{eqnarray}\label{lazard}
\Li_{S}(z)=\Sum_{n\ge0}\scal{S}{x_0^n}\Frac{\log^n(z)}{n!}
+\Sum_{k\ge1}\Sum_{w\in(x_0^*x_1)^kx_0^*}\scal{S}{w}\Li_w(z)
\end{eqnarray}
with $\CX\shuffle\serie{\C^{\mathrm{rat}}}{x_0}\shuffle\serie{\C^{\mathrm{rat}}}{x_1}
\subset\mathrm{Dom}(\Li_{\bullet})\subset\serie{\C^{\mathrm{rat}}}{X}$
and $\serie{\C^{\mathrm{rat}}}{X}$ denotes the closure, of ${\C}\pol{X}$ in $\CXX$,
by $\{+,.,{}^*\}$. For example \cite{IMACS,FPSAC96},
\begin{enumerate}
\item For any $i,j\in\N_+$ and $x\in X$, since $(t_0x_0+t_1x_1)^*=(t_0x_0)^*\shuffle(t_1x_1)^*$ 
and $(x^*)^{\minishuffle i}=(ix)^*$ then
$\Li_{(x_0^*)^{\minishuffle i}\shuffle(x_1^*)^{\minishuffle j}}(z)=z^i(1-z)^{-j}$.
\item For $a\in{\mathbb C},x\in X,i\in\N_+$, since $(ax)^{*i}=(ax)^*\shuffle(1+ax)^{i-1}$ then
\begin{eqnarray}
\Li_{(ax_0)^{*i}}(z)&=&z^a\Sum_{k=0}^{i-1}{i-1\choose k}\Frac{(a\log(z))^k}{k!},\\
\Li_{(ax_1)^{*i}}(z)&=&\Frac1{(1-z)^a}\Sum_{k=0}^{i-1}{i-1\choose k}\Frac{(a\log((1-z)^{-1})^k}{k!}.
\end{eqnarray}
\item Let $V=(t_1x_0)^{*s_1}x_0^{s_1-1}x_1\ldots(t_rx_0)^{*s_r}x_0^{s_r-1}x_1$,
for $(s_1,\ldots,s_r)\in\N_+^r$. Then
\begin{eqnarray}\label{zig}
\Li_V(z)=\sum_{n_1>\ldots>n_r>0}\frac{z^{n_1}}{(n_1-t_1)^{s_1}\ldots(n_r-t_r)^{s_r}}.
\end{eqnarray}
In particular, for $s_1=\ldots=s_r=1$, then one has
\begin{eqnarray}\label{zig}
\Li_{V}(z)
&=&\sum_{n_1,\ldots,n_r>0}\Li_{x_0^{n_1-1}x_1\ldots x_0^{n_r-1}x_1}(z)\;t_0^{n_1-1}\ldots t_r^{n_r-1}\cr
&=&\sum_{n_1>\ldots>n_r>0}\frac{z^{n_1}}{(n_1-t_1)\ldots(n_r-t_r)}.
\end{eqnarray}
\item From the previous points, one has
\begin{eqnarray}
\{\Li_S\}_{S\in\CX\shuffle\C[x_0^*]\shuffle\C[(-x_0^*)]\shuffle\C[x_1^*]}
&=&\mathrm{span}_\C\biggl\{\Frac{z^a}{(1-z)^b}\Li_w(z)\biggr\}_{w\in X^*}^{a\in\Z,b\in\N}\cr
&\subset&\mathrm{span}_{\C}\{\Li_{s_1,\ldots,s_r}\}_{s_1,\ldots,s_r\in\Z^r}\cr
&&\oplus\mathrm{span}_{\C}\{z^a|a\in\Z\},\\
\{\Li_S\}_{S\in\CX\shuffle\serie{\C^{\mathrm{rat}}}{x_0}\shuffle\serie{\C^{\mathrm{rat}}}{x_1}}
&=&\mathrm{span}_\C\biggl\{\Frac{z^a}{(1-z)^b}\Li_w(z)\biggr\}_{w\in X^*}^{a,b\in\C}\cr
&\subset&\mathrm{span}_{\C}\{\Li_{s_1,\ldots,s_r}\}_{s_1,\ldots,s_r\in\C^r}\cr
&&\oplus\mathrm{span}_{\C}\{z^a|a\in\C\}.
\end{eqnarray}
\end{enumerate}
\section{Noncommutative evolution equations}

As we said previously Drinfel'd proved that $(DE)$ admits two particular solutions on $\Omega$. These new tools and results can be considered as pertaining to the domain of \textit{noncommutative evolution equations}. We will, here, only mention what is relevant for our needs. 

Even for one sided \footnote{As the left (DE) for instance.} differential equations, in order to cope with limit initial conditions (see applications below), one needs the two sided version.   

Let then $\Omega\subset \C$ be simply connected and open and $\calH(\Omega)$ denote the algebra of holomorphic functions on $\Omega$. We suppose given two series (called \textit{multipliers}) $M_1,M_2\in \serie{\calH(\Omega)_+}{X}$ ($X$ is an alphabet and the subscript indicates that the series have no constant term). Let then
\begin{eqnarray*}
(DE_2)&&\d S=M_1S+SM_2.
\end{eqnarray*}
be our equation. 

\subsection{The main theorem}
\begin{theorem}\label{DE2}
Let 
\begin{eqnarray}
(DE_2)&&\d S=M_1S+SM_2.
\end{eqnarray}
\begin{enumerate}
\item[(i)] Solutions of $(DE_2)$ form a $\C$-vector space.
\item[(ii)] Solutions of $(DE_2)$ have their constant term (as coefficient of $1_{X^*}$) which are constant functions (on $\Omega$); there exists solutions with constant coefficient $1_{\Omega}$ (hence invertible).
\item[(iii)] If two solutions coincide at one point $z_0\in \overline{\Omega}$, they coincide everywhere.
\item[(iv)] Let be the following one-sided equations 
\begin{eqnarray}
(DE^{(1)})\quad \d S=M_1S \qquad (DE^{(2)})\quad \d S=SM_2.
\end{eqnarray}
and let $S_i,\ i=1,2$ a solution of $(DE^{(i)})$, then $S_1S_2$ is a solution of $(DE_2)$. Conversely, every solution of $(DE_2)$ can be constructed so.
\item[(v)] If $M_i,\ i=1,2$ are primitive and if $S$, a solution of $(DE_2)$, is group-like at one point, (or, even at one limit point) it is globally group-like. 
\end{enumerate}
\end{theorem}
\begin{proof}
Omitted.
\end{proof}
\begin{remark}
\begin{itemize}
\item Every holomorphic series $S(z)\in \serie{\calH(\Omega)}{X}$ which is group-like ($\Delta(S)=S\otimes S$ and $\scal{S}{1_{X^*}}$) is a solution of a left-sided dynamics with primitive multiplier (take $M_1=\d(S)S^{-1}$ and $M_2=0$).
\item Invertible solutions of an equation of type $S'=M_1S$ are on the same orbit by multiplication on the right by \textit{invertible constant series} i.e. let $S_i,\ i=1,2$ be invertible solutions of $(DE^{(1)})$, then there exists an unique invertible $T\in \serie{\C}{X}$ such that $S_2=S_1.T$. From this and point (iv) of the theorem, one can parametrize the set of invertible solutions of $(DE_2)$.     
\end{itemize}
\end{remark}

\subsection{Application: Unicity of solutions with asymptotic conditions.}

In a previous work \cite{Linz}, we proved that asymptotic group-likeness, for a series, implies\footnote{Under the condition that the multiplier be primitive, result extended as point (v) of the theorem above.} that the series in question is group-like everywhere. 
The process above (theorem \eqref{DE2}, Picard's process) can be performed, under certain conditions with improper integrals we then construct the series $\L$ recursively as 
$$
\scal{\L}{w}=\left\{
\begin{array}{rcl}
\Frac{\log^n(z)}{n!}&\mbox{if}&w=x_0^n\cr
\int_0^z\Big((\frac{x_1}{1-z})\scal{\L}{u}\Big)[s]\, ds&\mbox{if}&w=x_1u\cr
\int_0^z \Big((\frac{x_0}{z})\scal{\L}{ux_1x_0^n}\Big)[s]\, ds&\mbox{if}&w=x_0ux_1x_0^n.
\end{array}
\right.
$$

one can check that 
\begin{itemize}
\item this process is well defined at each step and computes the series $\L$ as below.  
\item $\L$ is solution of $(DE)$, is exactly $G_0$ and is group-like 
\end{itemize}
We here only prove that $G_0$ is unique using the theorem above. Consider the series 
$T=Le^{-x_0\log(z)}$. Then $T$ is solution of an equation of the type $(DE_2)$ 
\begin{equation}
T'=(\frac{x_0}{z}+\frac{x_1}{1-z})T+T(\frac{x_0}{z})
\end{equation}
but $\lim_{z\to z_0}G_0e^{-x_0\log(z)}=1$ so, by the point (iii) of theorem \eqref{DE2} one has 
$G_0e^{-x_0\log(z)}=Le^{-x_0\log(z)}$ and then $G_0=L$.  

A similar (and symmetric) argument can be performed for $G_1$ and then, in this interpretation and context, $\Phi_{KZ}$ is unique. 

\section{Double global regularization of associators}

Global singularities analysis leads to to the following global renormalization \cite{JSC}
\begin{eqnarray}
\Lim_{z\rightarrow1}\exp\biggl(-y_1\log\frac1{1-z}\biggr)\pi_Y(\L(z))
&=&\Lim_{n\rightarrow\infty}\exp\biggl(\Sum_{k\ge1}\H_{y_k}(n)\frac{(-y_1)^k}{k}\biggr)\H(n)\cr
&=&\pi_Y({Z}_{\minishuffle}).
\end{eqnarray}
Thus, the coefficients $\{\langle Z_{\shuffle}\vert u\rangle\}_{u\in X^*}$
({\it i.e.} $\{\zeta_{\shuffle}(u)\}_{u\in X^*}$) and
$\{\langle Z_{\stuffle}\vert v\rangle\}_{v\in Y^*}$
({\it i.e.} $\{\zeta_{\stuffle}(v)\}_{v\in Y^*}$) represent the finite part
of the asymptotic expansions,
in $\{(1-z)^{-a}\log^{b}(1-z)\}_{a,b\in\N}$ (resp. $\{n^{-a}\H_1^{b}(n)\}_{a,b\in\N}$)
of $\{\Li_w\}_{u\in X^*}$ (resp. $\{\H_w\}_{v\in Y^*}$).
On the other way, by a transfer theorem \cite{flajoletodlyzko},
let $\{\gamma_w\}_{v\in Y^*}$ be the finite parts of $\{\H_w\}_{v\in Y^*}$,
in $\{n^{-a}\log^{b}(n)\}_{a,b\in\N}$, and let $Z_{\gamma}$ be their noncommutative generating series.
The map $\gamma_{\bullet}:(\QY,\stuffle,1_{Y^*})\rightarrow({\cal Z},\times,1)$,
mapping $w$ to $\gamma_w$, is then a character and $Z_{\gamma}$ is group-like,
for $\Delta_{\stuffle}$. Moreover \cite{acta,VJM},
\begin{eqnarray}
Z_{\gamma}=\exp(\gamma y_1)\Prod_{l\in\Lyn Y-\{y_1\}}^{\searrow}\exp(\zeta(\Sigma_l)\Pi_l)
=\exp(\gamma y_1)Z_{\stuffle}.
\end{eqnarray}

The asymptotic behavior leads to the bridge\footnote{This equation is different from Jean \'Ecalle's one \cite{Ecalle3}.} equation \cite{JSC,acta,VJM}
\begin{eqnarray}\label{pont}
Z_{\gamma}=B(y_1)\pi_Y({Z}_{\minishuffle})
&\mbox{or equivalently}&
Z_{\shuffle}=B'(y_1)\pi_Y({Z}_{\minishuffle})
\end{eqnarray}
where $B(y_1)=\exp(\gamma y_1-\sum_{k\ge2}(-y_1)^k{\zeta(k)}/k)$
and $B'(y_1)=\exp(-\gamma y_1)B(y_1)$.

Similarly, there is $C^-_w\in\Q$ and $B^-_w\in\N$, such that
$\H^-_{w}(N){}_{\widetilde{N\rightarrow+\infty}}N^{(w)+\abs{w}}C^-_w$ and
$\Li^-_{w}(z){}_{\widetilde{z\rightarrow1}}(1-z)^{-(w)-\abs{w}}B^-_w$ \cite{Ngo}.
Moreover,
\begin{eqnarray}\label{CB}
C^-_w=\prod_{w=uv, v\neq 1_{Y_0^*}}((v)+\abs{v})^{-1}&\mbox{and}&B^-_w=((w)+\abs{w})!C^-_w.
\end{eqnarray}

Now, one can then consider the following noncommutative generating series~:
\begin{eqnarray}
\L^-:=\Sum_{w\in Y_0^*}\Li^-_ww,&\H^-:=\Sum_{w\in Y_0^*}\H^-_ww,&C^-:=\Sum_{w\in Y_0^*}C^-_ww.
\end{eqnarray}
Then $\H^-$ and $C^-$ are group-like for, respectively,
$\Delta_{\stuffle}$ and $\Delta_{\shuffle}$ and \cite{Ngo}
\begin{eqnarray}
&\lim\limits_{z\to1}h^{\odot-1}((1-z)^{-1})\odot\L^-(z)=\lim\limits_{N \to+\infty}g^{\odot-1}(N)\odot\H^-(N)=C^-,&\\
&h(t)=\Sum_{w\in Y_0^*}{((w)+\abs w)!}{t^{(w)+\abs w}}w\quad\mbox{and}\quad g(t)=\biggl(\Sum_{y\in Y_0}t^{(y)+1}y\biggr)^*.
\end{eqnarray}

Next, for any $w\in Y_0^*$, there exists then a unique polynomial
$p\in(\Z[t],\times,1)$ of degree $(w)+\abs{w}$ such that \cite{Ngo}
\begin{eqnarray}
\Li^-_w(z)=&\!\!\!\!\Sum_{k=0}^{(w)+\abs{w}}\frac{p_k}{(1-z)^k}
=\Sum_{k=0}^{(w)+\abs{w}}p_ke^{-k\log(1-z)}\!\!\!\!&\in(\Z[(1-z)^{-1}],\times,1),\label{Limoins0}\\
\H^-_w(n)=&\!\!\!\!\Sum_{k=0}^{(w)+\abs{w}}p_k{n+k-1\choose k-1}
=\Sum_{k=0}^{(w)+\abs{w}}\frac{p_k}{k!}(n)_k\!\!\!\!&\in(\Q[(n)_{\bullet}],\times,1),\label{Hmoins0}
\end{eqnarray}
where\footnote{Here, it is also convenient to  denote
$\Q[(n)_{\bullet}]$ the set of ``polynomials'' expanded as follows
\begin{eqnarray*}
\forall p\in,\Q[(n)_{\bullet}],
&p=\Sum_{k=0}^{d}p_k(n)_k,
&\deg(p)=d.
\end{eqnarray*}}
where $(n)_{\bullet}:\N\longrightarrow\Q$ mapping $i$ to $(n)_{i}=n(n-1)\ldots(n-i+1)$. 
In other terms, for any $w\in Y_0^*,k\in\N,0\le k\le(w)+\abs{w}$,
one has ${\scal{\Li^-_w}{(1-z)^{-k}}}={k!}\scal{\H^-_w}{(n)_k}$.


Hence, denoting $\tilde p$ the exponential transformed of the polynomial $p$, one has
$\Li^-_w(z)=p((1-z)^{-1})$ and $\H^-_w(n)=\tilde p((n)_{\bullet})$ with
\begin{eqnarray}
p(t)=\Sum_{k=0}^{(w)+\abs{w}}p_kt^{k}\in(\Z[t],\times,1)
&\mbox{and}&
\tilde p(t)=\Sum_{k=0}^{(w)+\abs{w}}\Frac{p_k}{k!}t^{k}\in(\Q[t],\times,1).\label{moins}
\end{eqnarray}
Let us then associate  $p$ and $\tilde p$ with the polynomial
$\check p$ obtained as follows
\begin{eqnarray}\label{pcheck}
\check p(t)=\Sum_{k=0}^{(w)+\abs{w}}k!p_kt^{k}=\Sum_{k=0}^{(w)+\abs{w}}p_kt^{\shuffle k}\in(\Z[t],\shuffle,1).
\end{eqnarray}

Let us recall also that, for any $c\in\C$,
one has $(n)_c{}_{\widetilde{n\rightarrow+\infty}}n^c=e^{c\log(n)}$ and,
with the respective scales of comparison, one has the following finite parts
\begin{eqnarray}
\mathrm{f.p.}_{z\rightarrow1}c\log(1-z)=0,&&\{(1-z)^a\log^b((1-z)^{-1})\}_{a\in\Z,b\in\N},\label{fp1}\\
\mathrm{f.p.}_{n\rightarrow+\infty}c\log n=0,&&\{n^a\log^b(n)\}_{a\in\Z,b\in\N}\label{fp2}.
\end{eqnarray}

Hence, using the notations given in \eqref{Limoins0} and \eqref{Hmoins0},
one can see, from \eqref{fp1} and \eqref{fp2}, that the values $p(1)$ and
$\tilde p(1)$ obtained in \eqref{moins} represent
\begin{eqnarray}
\mathrm{f.p.}_{z\rightarrow1}\Li^-_{w}(z)
=\mathrm{f.p.}_{z\rightarrow1}\Li_{R_w}(z)&=&p(1)\in\Z,\\
\mathrm{f.p.}_{n\rightarrow+\infty}\H^-_{w}(n)
=\mathrm{f.p.}_{n\rightarrow+\infty}\H_{\pi_Y(R_w)}(n)&=&\tilde p(1)\in\Q.
\end{eqnarray}
One can use then these values $p(1)$ and $\tilde p(1)$,
instead of the values $B^-_w$ and $C^-_w$, to regularize,
respectively, $\zeta_{\shuffle}(R_w)$ and $\zeta_{\gamma}(\pi_Y(R_w))$
as showed Theorem \ref{fin} bellow because, essentially,
$B^-_{\bullet}$ and $C^-_{\bullet}$ do not realize characters for, respectively,
$(\QX,\shuffle,1_{X^*},\Delta_{\shuffle},{\tt e})$ and
$(\QY,\stuffle,1_{Y^*},\Delta_{\stuffle},{\tt e})$ \cite{Ngo}.

Now, in virtue of the extension of $\Li_{\bullet}$,
defined as in \eqref{simple} and \eqref{lazard},
and of the Taylor coefficients, the previous polynomials
$p,\tilde p$ and $\check p$ given in \eqref{moins}--\eqref{pcheck}
can be determined explicitly thanks to

\begin{proposition}\label{explicit}
\begin{enumerate}
\item The following morphisms of algebras are {\em bijective} 
$$\begin{array}{lcccrcl}
\lambda:(\Z[x_1^*],\shuffle,1_{X^*})&\longrightarrow&(\Z[(1-z)^{-1}],\times,1),&R&\longmapsto&\Li_R,\\
\eta:(\Q[y_1^*],\stuffle,1_{Y^*})&\longrightarrow&(\Q[(n)_{\bullet}],\times,1),&S&\longmapsto&\H_S.
\end{array}$$
\item For any $w=y_{s_1},\ldots y_{s_r}\in Y_0^*$,
there exists a unique polynomial $R_w$ belonging to $(\Z[x_1^*],\shuffle,1_{X^*})$
of degree $(w)+\abs{w}$, such that
\begin{eqnarray*}
\Li_{R_w}(z)=\Li^-_w(z)=&p((1-z)^{-1})&\in(\Z[(1-z)^{-1}],\times,1),\\
\H_{\pi_Y(R_w)}(n)=\H^-_w(n)=&\tilde p((n)_{\bullet})&\in(\Q[(n)_{\bullet}],\times,1).
\end{eqnarray*}
In particular, via the extension, by linearity, of $R_{\bullet}$ over $\Q\pol{Y_0}$
and via the linear independent family $\{\Li^-_{y_k}\}_{k\ge0}$
in $\Q\{\Li^-_w\}_{w\in Y_0^*}$, one has
\begin{eqnarray*}
\forall k,l\in\N,&&
\Li_{R_{y_k}\shuffle R_{y_l}}=\Li_{R_{y_k}}\Li_{R_{y_l}}
=\Li^-_{y_k}\Li^-_{y_l}=\Li^-_{y_k\top y_l}=\Li_{R_{y_k\top y_l}}.
\end{eqnarray*}

\item For any $w$, one has $\check p(x_1^*)=R_w$.

\item More explicitly, for any $w=y_{s_1},\ldots y_{s_r}\in Y_0^*$,
there exists a unique polynomial $R_w$ belonging to $(\Z[x_1^*],\shuffle,1_{X^*})$
of degree $(w)+\abs{w}$, given by
\begin{eqnarray*}
R_{y_{s_1}\ldots y_{s_r}}
=\sum_{k_1=0}^{s_1}\sum_{k_2=0}^{s_1+s_2-k_1}\ldots
\sum_{k_r=0}^{(s_1+\ldots+s_r)-\atop(k_1+\ldots+k_{r-1})}
\binom{s_1}{k_1}\binom{s_1+s_2-k_1}{k_2}\ldots\cr
\binom{s_1+\ldots+s_r-k_1-\ldots-k_{r-1}}{k_r}
\rho_{k_1}\shuffle\ldots\shuffle\rho_{k_r},
\end{eqnarray*}
where, for any $i=1,\ldots,r$, if ${k_i}=0$ then $\rho_{k_i}=x_1^*-1_{X^*}$ else,
for ${k_i}>0$, denoting the Stirling numbers of second kind by  $S_2(k,j)$'s, one has
\begin{eqnarray*}
\rho_{k_i}=\Sum_{j=1}^{k_i}S_2({k_i},j)(j!)^2\sum_{l=0}^j
\Frac{(-1)^{l}}{l!}\frac{(x_1^*)^{\shuffle(j-l+1)}}{(j-l)!}.
\end{eqnarray*}
\end{enumerate}
\end{proposition}

\begin{proposition}[\cite{JSC,acta,VJM}]\label{to0}
With notations of \eqref{zeta}, similar to the character $\gamma_{\bullet}$,
the poly-morphism $\zeta$ can be extended as follows
\begin{eqnarray*}
\zeta_{\shuffle}:(\QX,\shuffle,1_{X^*})\longrightarrow({\cal Z},\times,1),&&
\zeta_{\stuffle}:(\QY,\stuffle,1_{Y^*})\longrightarrow({\cal Z},\times,1)
\end{eqnarray*}
satisfying, for any $l\in\Lyn Y-\{y_1\}$,
$\zeta_{\shuffle}(\pi_X(l))=\zeta_{\stuffle}(l)=\gamma_{l}=\zeta(l)$
and, for the generators of length (resp. weight) one, for $X^*$ (resp. $Y^*$),
$\gamma_{y_1}=\gamma$ and $\zeta_{\shuffle}(x_0)=\zeta_{\shuffle}(x_1)=\zeta_{\stuffle}(y_1)=0$.
\end{proposition}

Now, to regularize  $\{\zeta(s_1,\ldots,s_r)\}_{(s_1,\ldots,s_r)\in\C^r}$, we use

\begin{lemma}[\cite{Ngo}]
\begin{enumerate}
\item The power series $x_0^*$ and $x_1^*$ are transcendent over $\CX$.

\item The family $\{x_0^*,x_1^*\}$ is algebraically independent over
$(\CX,\shuffle,1_{X^*})$ within $(\CXX,\shuffle,1_{X^*})$.

\item The module $(\CX,\shuffle,1_{X^*})[x_0^*,x_1^*,(-x_0)^*]$ is $\CX$-free
and the family
$\{(x_0^*)^{\shuffle k}\shuffle (x_1^*)^{\shuffle l}\}^{(k,l)\in\Z\times\N}$
forms a $\CX$-basis of it.

Hence, $\{w\shuffle(x_0^*)^{\shuffle k}\shuffle(x_1^*)^{\shuffle l}\}^{(k,l)
\in\Z\times\N}_{w\in X^*}$ is a $\C$-basis of it.

\item One has, for any $x_i\in X$,
$\serie{\C^{\mathrm{rat}}}{x_i}=\mathrm{span}_{\C}\{(tx_i)^*\shuffle{\C}\pol{x_i}\vert{t\in\C}\}$.
\end{enumerate}
\end{lemma}

Since, for any $t\in\C,\abs{t}<1$, one has
$\Li_{(tx_1)^*}(z)=(1-z)^{-t}$ and
\begin{eqnarray}
\H_{\pi_Y(tx_1)^*}=\Sum_{k\ge0}\H_{y_1^k}t^k=
\exp\biggl(-\Sum_{k\ge1}\H_{y_k}\Frac{(-t)^k}{k}\biggr)
\end{eqnarray}
then, with the notations of Proposition \ref{to0},
we extend extend the characters $\zeta_{\shuffle}$ and $\gamma_{\bullet}$,
defined in Proposition \ref{to0}, over $\CX\shuffle\C[x_1^*]$
and $\CY\stuffle\C[y_1^*]$, respectively, as follows

\begin{proposition}[\cite{Ngo}]\label{reg}
The characters $\zeta_{\shuffle}$ and $\gamma_{\bullet}$ can be extended as follows
$$\begin{array}{rcl}
\zeta_{\shuffle}:(\CX\shuffle\C[x_1^*],\shuffle,1_{X^*})&\longrightarrow&(\C,\times,1_{\C}),\\
\forall t\in\C,\abs{t}<1,\quad(tx_1)^*&\longmapsto&1_{\C}.\\
\gamma_{\bullet}:(\CY\stuffle\C[y_1^*],\stuffle,1_{Y^*})&\longrightarrow&(\C,\times,1_{\C}),\\
\forall t\in\C,\abs{t}<1,\quad(ty_1)^*&\longmapsto&\exp\biggl(\gamma t-\Sum_{n\ge2}\zeta(n)\Frac{(-t)^n}{n}\biggr)=\frac1{\Gamma(1+t)}.
\end{array}$$
\end{proposition}

Therefore, in virtue of Propositions \ref{explicit} and \ref{reg}, we obtain then
\begin{theorem}\label{fin}
\begin{enumerate}
\item For any $(s_1,\ldots,s_r)\in\N_+^r$ associated with $w\in Y^*$, there exists
a unique polynomial $p\in\Z[t]$ of valuation $1$ and of degree $(w)+\abs{w}$ such that 
$$\begin{array}{rcll}
\check p(x_1^*)&=&R_w&\in(\Z[x_1^*],\shuffle,1_{X^*})\\
p((1-z)^{-1})&=&\Li_{R_w}(z)&\in(\Z[(1-z)^{-1}],\times,1),\\
\tilde p((n)_{\bullet})&=&\H_{\pi_Y(R_w)}(n)&\in(\Q[(n)_{\bullet}],\times,1),\\
\zeta_{\shuffle}(-s_1,\ldots,-s_r)=p(1)&=&\zeta_{\shuffle}(R_w)&\in(\Z,\times,1),\\
\gamma_{-s_1,\ldots,-s_r}=\tilde p(1)&=&\gamma_{\pi_Y(R_w)}&\in(\Q,\times,1).
\end{array}$$

\item Let $\Upsilon(n)\in\serie{\Q[(n)_{\bullet}]}{Y}$
and $\Lambda(z)\in\serie{\Q[(1-z)^{-1}][\log(z)]}{X}$
be the noncommutative generating series of $\{\H_{\pi_Y(R_w)}\}_{w\in Y^*}$
and $\{\Li_{R_{\pi_Y(w)}}\}_{w\in X^*}$~:
\begin{eqnarray*}
\Upsilon:=\Sum_{w\in Y^*}\H_{\pi_Y(R_w)}w\mbox{ and }
\Lambda:=\Sum_{w\in X^*}\Li_{R_{\pi_Y(w)}}w,\mbox{ with }
\scal{\Lambda(z)}{x_0}=\log(z).
\end{eqnarray*}
Then $\Upsilon$ and $\Lambda$ are group-like,
for respectively $\Delta_{\stuffle}$ and $\Delta_{\shuffle}$, and~:
\begin{eqnarray*}
\Upsilon=\Prod_{l\in\Lyn Y}^{\searrow}e^{\H_{\pi_Y(R_{\Sigma_l})}\Pi_l}
&\mbox{and}&
\Lambda=\Prod_{l\in\Lyn X}^{\searrow}e^{\Li_{R_{\pi_Y(S_l)}}P_l}.
\end{eqnarray*}

\item Let $Z^-_{\gamma}\in\serie{\Q}{Y}$ and $Z^-_{\shuffle}\in\serie{\Z}{X}$
be the noncommutative generating series of $\{\gamma_{\pi_Y(R_w)}\}_{w\in Y^*}$
and\footnote{On the one hand, by Proposition \ref{to0},
one has $\scal{Z^-_{\shuffle}}{x_0}=\zeta_{\shuffle}(x_0)=0$.

On the other hand, since $R_{y_1}=(2x_1)^*-x_1^*$ then
$\Li_{R_{y_1}}(z)=(1-z)^{-2}-(1-z)^{-1}$ and
$\H_{\pi_Y(R_{y_1})}(n)={n\choose2}-{n\choose1}$.
Hence, one also has
$\scal{Z^-_{\shuffle}}{x_1}=\zeta_{\shuffle}(R_{\pi_Y(y_1)})=0$
and $\scal{Z^-_{\gamma}}{x_1}=\gamma_{\pi_Y(R_{y_1})}=-1/2$.}
$\{\zeta_{\shuffle}(R_{\pi_Y(w)})\}_{w\in X^*}$, respectively~:
\begin{eqnarray*}
Z^-_{\gamma}:=\Sum_{w\in Y^*}\gamma_{\pi_Y(R_w)}w&\mbox{and}&
Z^-_{\shuffle}:=\Sum_{w\in X^*}\zeta_{\shuffle}(R_{\pi_Y(w)})w.
\end{eqnarray*}
Then $Z^-_{\gamma}$ and $Z^-_{\shuffle}$ are group-like,
for respectively $\Delta_{\stuffle}$ and $\Delta_{\shuffle}$, and~:
\begin{eqnarray*}
Z^-_{\gamma}=\Prod_{l\in\Lyn Y}^{\searrow}e^{\gamma_{\pi_Y(R_{\Sigma_l})}\Pi_l}
&\mbox{and}&
Z^-_{\shuffle}=\Prod_{l\in\Lyn X}^{\searrow}e^{\zeta_{\shuffle}(\pi_Y(S_l))P_l}.
\end{eqnarray*}
\end{enumerate}
\end{theorem}
Moreover, $\mathrm{F.P.}_{n\rightarrow+\infty}\Upsilon(n)=Z^-_{\gamma}$
and $\mathrm{F.P.}_{z\rightarrow1}\Lambda(z)=Z^-_{\shuffle}$ meaning that,
for any $v\in Y^*$ and $u\in X^*$, one has
\begin{eqnarray}
\mathrm{f.p.}_{n\rightarrow+\infty}\scal{\Upsilon(n)}{v}=\scal{Z^-_{\gamma}}{v}
&\mbox{and}&\mathrm{f.p.}_{z\rightarrow1}\scal{\Lambda(z)}{u}=\scal{Z^-_{\shuffle}}{u}.
\end{eqnarray}

\end{document}